\documentclass[11pt]{amsart}

\usepackage[OT2,T1]{fontenc}
\DeclareSymbolFont{cyrletters}{OT2}{wncyr}{m}{n}
\DeclareMathSymbol{\Sha}{\mathalpha}{cyrletters}{"58}
\usepackage[utf8]{inputenc}
\usepackage{musicography}
\usepackage{amsmath}

\usepackage{array}
\usepackage{accents}

\usepackage{tikz-cd}
\usepackage{hyperref}
\usepackage{verbatim}
\usepackage{amssymb}
\usepackage{textcomp}
\usepackage{amsthm}
\usepackage{amsfonts}
\usepackage[all]{xy}
\usepackage{mathrsfs}
 \usepackage{graphicx}
 \usepackage{epstopdf}
 \usepackage{float}
 \usepackage{graphicx}
 \usepackage[T1]{fontenc}
\newtheorem{theo}{Th\'eor\'eme}[section]
\usepackage[ddmmyyyy]{datetime}

\newcolumntype{L}[1]{>{\raggedright\arraybackslash}m{#1}}

\def\tilde{\widetilde}

\def\remk{\noindent\textit{Remark:~}}
\numberwithin{equation}{section}

\newtheorem{prop}[theo]{Proposition}
\newtheorem{def-prop}[theo]{Definition-Proposition}
\newtheorem{cor}[theo]{Corollary}

\newtheorem{lemma}[theo]{Lemma}
\newtheorem{teo}[theo]{Theorem}
\newtheorem{definition}[theo]{Definition}

\newtheorem{remark}[theo]{Remark}

\newtheorem{assumption}[theo]{Assumption}

\def\A{\mathbb{A}}

\def\C{\mathbb{C}}

\def\G{\mathbb{G}}

\def\Q{\mathbb{Q}}
\def\R{\mathbb{R}}

\def\Z{\mathbb{Z}}

\def\k{\mathrm{k}}

  \newcommand{\cC}{{\mathcal C}}

    \newcommand{\cH}{{\mathcal H}}

 \newcommand{\cM}{{\mathcal M}}
 
 \def\cO{\mathcal{O}}

\newcommand{\cS}{{\mathcal S}}

 \newcommand{\bL}{{\mathbb L}}

\newcommand{\sH}{{\mathscr H}}

 \def\Sym {\mathop{\mathrm{Sym}}\nolimits}
 
 \def\log {\mathop{\mathrm{log}}\nolimits}
 
 \def\vol {\mathop{\mathrm{vol}}\nolimits}

  \def\triv {\mathop{\mathrm{triv}}\nolimits}

  \def\op {\mathop{\mathrm{op}}\nolimits}

\def\supp{\mathop{\mathrm{supp}}\nolimits}

\def\wt{\mathop{\mathrm{wt}}\nolimits}
\def\mod{\mathop{\mathrm{mod}}\nolimits}

\def\GL{\mathop{\mathrm{GL}}\nolimits}

\def\Id{\mathop{\mathrm{Id}}\nolimits}
\def\hat{\widehat}

\allowdisplaybreaks

\everymath{\displaystyle}

\newlength{\dhatheight}

\makeatletter
\renewcommand\section{\@startsection{section}{1}{\z@}%
                       {-18\p@ \@plus -4\p@ \@minus -4\p@}%
                       {12\p@ \@plus 4\p@ \@minus 4\p@}%
                       {\normalfont\large\bfseries\boldmath
                        \rightskip=\z@ \@plus 8em\pretolerance=10000 }}
\renewcommand\subsection{\@startsection{subsection}{2}{\z@}%
                       {-18\p@ \@plus -4\p@ \@minus -4\p@}%
                       {8\p@ \@plus 4\p@ \@minus 4\p@}%
                       {\normalfont\normalsize\bfseries\boldmath
                        \rightskip=\z@ \@plus 8em\pretolerance=10000 }}
\renewcommand\subsubsection{\@startsection{subsubsection}{3}{\z@}%
                       {-18\p@ \@plus -4\p@ \@minus -4\p@}%
                       {4\p@ \@plus 2\p@ \@minus 2\p@}%
                       {\normalfont\normalsize\bfseries\boldmath
                        \rightskip=\z@ \@plus 8em\pretolerance=10000 }}
\renewcommand\paragraph{\@startsection{paragraph}{4}{\z@}%
                       {-12\p@ \@plus -4\p@ \@minus -4\p@}%
                       {2\p@ \@plus 1\p@ \@minus 1\p@}%
                       {\normalfont\normalsize\itshape
                        \rightskip=\z@ \@plus 8em\pretolerance=10000 }}
\makeatother

\usepackage[acronym,description,nopostdot]{glossaries}

\newglossaryentry{centralizer}%
{%
  name={\ensuremath{S_\phi}},
  description={centralizer of the parameter $\phi$}
}

\makenoidxglossaries

\input xy
\xyoption{all}

\title{Poisson summation for Hankel transforms}
\author{Taiwang DENG}
\address{ 
Yau Mathematical Sciences Center, Tsinghua University, Haidian District, Beijing, 100084, China.}
\email{dengtaiw@tsinghua.edu.cn}
\date{}


\begin{document}

\maketitle

\begin{abstract}
In this article we study the Poisson summation for Hankel transform in the sense of Braverman-Kazhdan-Ngo in the
special case of $L$-embedding $\rho: \GL_1\rightarrow \GL_2$. We view such a summation formula as the generalization of the classical Voronoi summation
formula.
\end{abstract}

\tableofcontents

\section{Introduction}
The Voronoi summation formula proved by Voronoi is 
the following: Let $d(n)=\sum_{d|n} 1$ be the classical divisor 
function and $\varphi: \R\rightarrow \C$ be a smooth
compactly supported function in $(0, \infty)$. Then
\[
\sum_{n=1}^\infty d(n)\varphi(n)=\int_{0}^\infty \varphi(x)(\log x+2 \gamma) dx+\sum_{n=1}^\infty d(n)\hat{\varphi}(n), 
\]
where $\gamma$ is the Euler-Mascheroni constant and the function $\hat{\varphi}$ is the Hankel transform given by
\[
\hat{\varphi}(x)=\int_{0}^\infty \varphi(y)(4K_0(4\pi |xy|^{1/2})-2\pi Y_0(4\pi|xy|^{1/2})dy, 
\]
where $K_0$ and $Y_0$ are Bessel functions. Later in 1927, Oppenheim gave the following generalization
for $\sigma_{s}=\sum_{d|n} d^s$ and $1/4<\Re(s)<3/4$, 
\begin{align*}
&\sum_{n=1}^{\infty}\sigma_{2s-1}(n)n^{1/2-s}\varphi(n)\\
&=\sum_{n=1}^{\infty}\sigma_{2s-1}(n)n^{1/2-s}\cH_s{\varphi}(n)\\
&+\int_{0}^{\infty}\varphi(x)(\zeta_\Q(2s)x^{s-1/2}+\zeta_\Q(2-2s)x^{1/2-s})dx
\end{align*}
where
\begin{align*}
\cH_s{\varphi}(x)&=\int_{0}^{\infty}\varphi(y)[-2\pi \cos(\pi s)J_{1-2s}(4\pi\sqrt{xy})-2\pi \sin(\pi s)Y_{1-2s}(4\pi \sqrt{xy})\\
&+4\sin(\pi s)K_{1-2s}(4\pi \sqrt{xy})]dy
\end{align*}
In particular, letting $s\rightarrow 1/2$ we obtain the original Voronoi formula.

Following Beineke and Bump \cite[Discussion after Proposition 6]{Bump06}, we can write the above transform in the following way:
although the integral 
\[
\int_{\R}|w|^{-2s}\psi(-w-w^{-1}a)dw
\]
is never absolutely convergent, it is conditionally convergent if $0<\Re(s)<1/2$.
Moreover, if we take $\psi(x)=e^{2\pi i x}$, then
\begin{align*}
&\int_{\R}|w|^{-2s}\psi(-w-w^{-1}a)dw\\
&=\left\{\begin{array}{cc}\\
 &-2\pi a^{1/2-s}[\cos(\pi s)J_{1-2s}(4\pi\sqrt{a})+\sin(\pi s)Y_{1-2s}(4\pi \sqrt{a})], \text{ if }a>0;\\
&4|a|^{1/2-s}\sin(\pi s)K_{1-2s}(4\pi \sqrt{a}), \text{ if } a<0.
\end{array}\right.
\end{align*}
Letting $s\rightarrow 1/2$, we get(this is indeed valid in distribution sense) 
\[
K(x):=\int_{\R^\times}\psi(-w-w^{-1}x)d^\times w=4K_0(4\pi |xy|^{1/2})-2\pi Y_0(4\pi|xy|^{1/2}),  
\]
where $d^\times x=\frac{dx}{|x|}$. Note that the kernel function appears in the work of 
Ngo \cite[\S 6]{Ngo20} which is termed $\rho$-Bessel function for the representation of $L$-groups:
\[
\rho: \C^\times\rightarrow \GL_2(\C), \quad t\mapsto \begin{bmatrix}t&0\\ 0&t\end{bmatrix}.
\]
This is our starting point of developing an adelic version of the Voronoi summation formula 
as a Poisson summation formula in the sense of Braverman-Kazhdan-Ngo.

More precisely, we consider the Hankel transform associated to the adelic $\rho$-Bessel function
\[
K_{\A_F}(x, s)= \int_{\A_F^\times }|w|^{1-2s}\psi(-w-w^{-1}x)d^\times w.
\]
We introduce the $\rho$-Schwartz spaces(local and global) 
in our setup and show that they are stable under  Hankel transform. Our main result is Theorem \ref{teo-Poisson-summation-GL_1}, 
which is both a Poisson summation in the sense of Braverman-Kazhdan-Ngo and an adelic generalization of the 
Voronoi summation formula. We will continue along this line to 
investigate the Poisson summation as well as the generalization of the Voronoi summation formula
in later work. Also, a general Poisson summation formula is established for representations of torus in \cite{Laff16}, 
we hope to address the relation to our approach in future work.

\par \vskip 1pc
{\bf Acknowledgement.}
The work is done when the author is a postdoc at Yau Mathematical Sciences Center
of Tsinghua University. He wants to thank their hospitality. He also would like to thank Bin Xu for many
discussions on related subjects.

\section{Fourier analysis on L-monoid}

In this section, let $F$ be a finite extension of $\Q_p$.
Let $\cO_F$ be the ring of integers of $F$ and $\pi_F$ be its uniformizer. Let $\k_F$ be its residue field such that $q=|\k_F|$.
We fix a valuation $\nu_{\pi_F}$ on $F$ with $\nu_{\pi_F}(\pi_F)=1$.

Let $G$ be a quasi-split reductive group over a local field $\cO_F$. Let $\rho: {^L}G\rightarrow \GL(V_\rho)$ be a representation of ${^L}G$ such that the image of 
$\hat{G}$ contains the center of $\GL(V_\rho)$(equivalently, there is a character of $G$ whose dual induces an central action of $\C^\times$ by scalar on $V_\rho$). From the map $\rho$ Ngo(\cite[\S 5.1]{Ngo20}) constructed an $L$-Monoid
$M_G$ over $F$ with a non-trivial morphism 
\[
\nu_G: M_\rho\rightarrow \A^1.
\]
Moreover, the character(which is $\Gamma_F$-fixed)
\[
\begin{tikzcd}
\hat{\mu}_G: \hat{G}\arrow[r] &\GL(V_\rho)\arrow[r, "\det"] &\C^\times
\end{tikzcd}
\]
gives rise to a central cocharacter $\mu_G: \G_m\rightarrow G$ over $F$.
We make the following assumption.

\begin{assumption}
The monoid $M_\rho$, $\nu_G$ and $\mu_G$ admit a model over $\cO_F$(we keep the same notation for their models). And $\nu_G$ is smooth over the open scheme $\G_m=\A^1\backslash \{0\}$.
\end{assumption}

\begin{teo}[{\rm Weil}]\label{teo-Weil-measure}
Let $X/\cO_F$ be a smooth scheme of 
relative dimension $d$.
There exists a canonical measure $\mu_{can}$
on the set $X(\cO_F)$, such 
that for a $k$-rational point $x\in X(k_F)$ we have 
\[
\vol(e^{-1}(x))=\frac{1}{q^d},
\]
where $e: X(\cO_F)\rightarrow X(k)$ is the
specialization map.
In particular, one has $\vol(X(\cO_F))=\frac{X(k)}{q^d}$.
\end{teo}
The above theorem can also 
be generalized to Deligne-Mumford stacks
(cf. \cite{Groe20}, \cite{Yasu17}).

Let $|\cdot|$ be the absolute value on $F$ induced by $\nu$ such that $|\pi_F|=\frac{1}{q}$. We denote by $\cC_c^\infty(M_\rho)$ the space of compactly supported and locally constant functions on $M_\rho(F)$.
Similarly let $\cC_c^\infty(G)$ be the space of compactly supported and locally constants functions on $G(F)$.

By our assumption, $M_\rho\backslash\nu_G^{-1}(0)$ is smooth over $\cO_F$ hence by Weil's theorem we can attach
a canonical measure to it, which is denoted by $d\mu$. We extend $d\mu$ to $M_\rho(\cO_F)$ by letting the closed subset 
$\nu_G^{-1}(0)$ to be of measure zero. We further need to extend the measure to $M_\rho(F)$. To do that, we refer to $\mu_G$.
In fact, the composition $\nu_G\circ \mu_G: \G_m\rightarrow \G_m,  t\mapsto t^{n_\rho}$ is a algebraic character which is non-trivial by 
our assumption. Following Ngo, we know that 
\[
n_\rho=\langle 2\eta_G, \lambda_\rho\rangle +1
\]
where $\lambda_\rho$ is the highest weight of $\rho$ and $2\eta_G$ is the sum of positive roots of $G$.
Anyway, the cocharacter $\mu_G$ induces an action of $F^\times$ on $M_\rho(F)$.  Finally we have to understand the change of $d\mu$
under the action of $F^\times$.

\begin{lemma}
We have $d(\pi_F\mu)=|\pi_F|^{n_\rho r_\rho} d\mu$ for some positive number $r_\rho$.
\end{lemma}
\begin{proof}
This follows from the definition of Weil measures.
\end{proof}

\remk For $G=\{(\lambda, g)\in \GL_1\times\GL_2| \lambda^n=\det(g)\}$ with 
\[
\rho: \hat{G}\rightarrow \Sym^n(V), \quad V=\C^2
\]
with $\hat{G}=\GL_1\times \GL_2/\G_m$, where $\G_m$ is identified with the subgroup $\{(t^n, t^{-1}\Id)\}$ of $\GL_1\times \GL_2$
and $\rho$ corresponds to the symmetric power representation of $\GL_2$ extending to $\hat{G}$ via scalar multiplication 
on the factor $\GL_1$. In this case $r_\rho=2(n_\rho-1)=2n$ with $n_\rho=n+1$.  We also write down the 
corresponding $\nu_G$ and $\mu_G$
\begin{align*}
\nu:G&\rightarrow \G_m, \quad (\lambda, g)\mapsto \lambda\\
\mu_G: \G_m&\rightarrow G, \quad  t\mapsto (t^{n_\rho}, t^{\frac{n_\rho(n_\rho-1)}{2}}\Id).
\end{align*}

Let $\nu_s(X)=|\nu(X)|^s$.  Fix $f_{\rho}\in \cC_c^\infty(M_\rho)$ with the following properties:
\begin{itemize}
\item[(1)] it is supported on $M_\rho(\cO_F)$;
\item[(2)]its restriction to the fibers of $\nu_s$ is constant(so that $f_\rho$ is unramified);
\item[(3)]the integration $\int_{M_\rho(\cO_F)}f_\rho \nu_sd\mu$
is a rational function in $q^{-s}$.
\end{itemize}

\begin{definition}
We define for $f\in \cC^\infty(M_\rho)$, 
\[
\epsilon(f)(X)=f(\pi_F^{-1}X), \forall X\in M_\rho(F).
\]
Furthermore, for $P\in \C[\epsilon]$ with $P(\epsilon)=\sum_i a_i\epsilon^i$, we define
\[
P(f)=\sum_i a_i\epsilon^i(f).
\]
\end{definition}

\begin{prop}\label{prop-polynomial-basic-compactly-supported}
There exists $P\in \C[\epsilon]$ such that 
\[
P(f_\rho)\in \cC_c^\infty(M_\rho)
\]
\end{prop}

\begin{proof}
We first consider the integration
\[
L(q^{-s})=\int_{M_\rho(F)}f_\rho(X)\nu_s(X)d\mu(X), 
\]
where $L(X)\in \C(X)$ is a rational function.
By our assumption that $f_\rho$ is constant along the fiber of $\nu_1$, we have
\[
\int_{M_\rho(F)}f_\rho(X)\nu_s(X)d\mu(X)=\sum_{i=0}^{\infty} c_i\vol(\nu_1^{-1}(q^i))q^{-is}, \quad c_i=f_\rho|_{\nu_1^{-1}(q^i)}.
\]
Moreover, we have 
\begin{align*}
\int_{M_\rho(F)}\epsilon(f_\rho)(X)\nu_s(X)d\mu(X)&=\int_{M_\rho(F)}f_\rho(\pi_F^{-1}X)\nu_s(X)d\mu(X)\\
&=\int_{M_\rho(F)}f_\rho(X)\nu_s(\pi_F X)d\pi_F\mu(X)\\
&=q^{-n_\rho (s+r_\rho)}\int_{M_\rho(F)}f_\rho(X)\nu_s(X)d\mu(X)\\
&=q^{-n_\rho (s+r_\rho)}L(q^{-s}).
\end{align*}
Similarly, 
\[
\int_{M_\rho(F)}\epsilon^n(f_\rho)(X)\nu_s(X)d\mu(X)=q^{-nn_\rho (s+r_\rho)}L(q^{-s})
\]

By assumption, 
\[
L(q^{-s})=\frac{R_1(q^{-s})}{R_2(q^{-s})}, \quad R_i(q^{-s})\in \C[q^{-s}](i=1,2)
\]
and we further have
\[
R_2(q^{-s})=\prod_i(\alpha_i-q^{-s}).
\]
Consequently, for $P\in \C[\epsilon]$, 
\[
\int_{M_\rho(F)}P(f_\rho)(X)\nu_s(X)d\mu(X)=P(q^{-n_\rho (s+r_\rho)})L(q^{-s}).
\]
It follows that there exists $P\in \C[\epsilon]$ such that $P(q^{-n_\rho (s+r_\rho)})L(q^{-s})\in \C[q^{-s}]$.
Since $P(f_\rho)(X)$ is also constant along fibers of $\nu_1$, we have 
\[
\int_{M_\rho(F)}P(f_\rho)(X)\nu_s(X)d\mu(X)=\sum_{i=0}^{\infty} d_i\vol(\nu_1^{-1}(q^i))q^{-is}, \quad d_i=P(f_\rho)|_{\nu_1^{-1}(q^i)}.
\]
It follows that $P(f_\rho)(X)$ is compactly supported with support contained in $M_\rho(\cO_F)$.

\end{proof}

\section{Hankel transform and its basic properties}

In this section, we fix a number field $F$. For $v$ a non-archimedean place of $F$, let $\cO_v$ be  the ring of integers of the local field $F_v$.
We also fix a uniformizer $\pi_v$ of $\cO_v$. 
We denote $q_v=|\cO_v/\pi_v\cO_v|$.
Let $\nu_{v}: F_v^\times \rightarrow \Z$ be a valuation with $\nu_v(\pi_v)=1$
and its associated norm $|x|_v=q_v^{-\nu_v(x)}$. And for $x=(x_v)\in \A_F$, define
\[
|x|=\prod_v|x_v|_v.
\]
Also, for $v$ archimedean let $\mathfrak{D}_v=(\omega_v)$ denote the different of $F_v$ over $\Q_p$.

Let us consider the generalized Poisson summation formula on $L$-monoid.

We begin with our settup: $V_\rho=\C^2$ and $\rho: \G_m\rightarrow \GL(V_\rho)$ such that $t\mapsto \begin{bmatrix}t&0\\ 0&t\end{bmatrix}$.

In this case the L-monoid is $M_\rho=\A^1$ the affine line (see \cite[\S 5]{Ngo20} for a construction ), and 
\[
L(s, \triv, \rho)=\zeta_F(s)^2, 
\]
where $\triv$ denotes the trivial representation of $\A_F^\times$, the idele group of $F$.

Following Ngo, 
for a non-archimedean place $v$, we define the basic function $f_{v}$ to be
\begin{align*}
f_{v}(x)=\left\{\begin{array}{cc}
&\nu_v(x)+1, \quad \text{ if } x\in \cO_v;\\
&0, \quad \text{ otherwise.}
\end{array}\right.
\end{align*}
It satisfies the property that  for any smooth character: $\chi: F_v^\times\rightarrow \C^\times$ 
\begin{align*}
\zeta_{F_v}(1)\int_{F_v}f_v(x)\chi(x) d^\times x=L(s, \chi, \rho)=
\left\{\begin{array}{ll}
&\frac{1}{(1-\chi(\pi_v))^2}, \text{ if $\chi$ is unramified; }\\
&0, \quad \text{    otherwise. }
\end{array}\right.
\end{align*}
In particular, 
\[
\zeta_{F_v}(1)\int_{F_v^\times}f_{v}(x)|x|_v^{s} d^\times x=\frac{1}{(1-q_v^{-s})^2}=L_p(s, \triv, \rho).
\]

Let us fix some notations for later use. We follow more or less \cite{Wri85}:
\begin{itemize}
\item[(1)]$\psi_{\A_F}:=\otimes_{v}\psi_v: \A_F\rightarrow \C^\times$ is a fixed  self dual additive characters and $d^\times w$ is the product of local measures 
$d^\times w= \otimes_v d^\times w_v $ such that 
\begin{align*}
d^\times w_v= \frac{d w_v}{|w_v|_v}.
\end{align*}
Moreover, when $v$ is non-archimedean, we normalize $\vol(\cO_F, dw)=1$, 
where $\cO_F$ is the ring of integers.  If $F_v=\R$, we take $\psi_v(x)=e^{2\pi ix}$ and if $F_v=\C$, take $\psi_v(x)=e^{4\pi i\Re(x)}$.
\item[(2)]$\nu: \A_F\rightarrow \GL_2(\A_F), x\mapsto \begin{bmatrix}1& x\\ 0&1\end{bmatrix}$;
\item[(3)] for $g\in \GL_2(\A_F)$, the Iwasawa decomposition gives 
\[
g=ka(t_1, t_2)n(c), \quad k\in K, a(t_1, t_2)=\begin{bmatrix}t_1& 0\\ 0&t_2\end{bmatrix}, n(c)=\begin{bmatrix}1& 0\\ c&1\end{bmatrix}, 
\]
where $K$ is a maximal compact subgroup of $\GL_2(\A_F)$. We define $t(g)=|\frac{t_1}{t_2}|_v^{1/2}$.
\end{itemize}

Note that here we can choose $\psi_v$ such that its conductor is $e_v$(by definition $\mathfrak{D}_v=(\omega_v)=(\pi_v^{e_v})$ is the different of $F_v$), 
hence  $\psi_v$ is trivial on $\pi^{-e_v}\cO_v$. Then self duality implies
\[
\vol(\cO_v, dx)=q_v^{-\frac{e_v}{2}}
\]
In this case we define the local zeta factor to be(cf. \cite[\S 1]{Wri86b})
\[
\zeta_{F_v}(s)=q_v^{\frac{e_vs}{2}}\frac{1}{1-q_v^{-s}}.
\]
For achimedean places take $\zeta_{\R}(s)=\pi^{-s/2}\Gamma(s/2)$ and $\zeta_{\C}(s)=(2\pi)^{1-s}\Gamma(s)$.
\\

Let us denote by $ \cS(F_v)$ the space of usual Schwartz functions and $ \cS(F_v)'$ the space of 
tempered distributions.
 
 \begin{definition}
 Let us define $\cS(F_v^\times)$ to be the restriction of $ \cS(F_v)$ to $F_v^\times$.
 \end{definition}
 
 \remk Let us explain the definition in case $F_v=\R$, then $f\in \cS(\R^\times)$ means that 
 there exists $\rho_{+}, \rho_{-}\in \cS(\R)$ such that 
 \[
f(x)= \left\{\begin{array}{cc}
 &\rho_{+}(x), \text{ if } x>0;\\
 &\rho_{-}(x), \text{ if } x<0.
 \end{array}\right.
 \]
 
\begin{definition}
Locally, we define $\cS_\rho(F_v^\times )$ to be the space 
$\cS(F_v^\times)+|x|_v^{1-2s}\cS(F_v^\times)$  if $1-2s\neq 0$ and $\cS(F_v^\times)+ \log(|x|_v)\cS(F_v^\times)$ otherwise
Globally, we define
$\cS_\rho(\A_F^\times)$ to be the restricted tensor product of the local $\rho$-Schwartz space: it is generated by 
\[
f=\otimes f_v, \quad f_v \text{ is basic for all but finitely many } v.
\]
\end{definition}

\remk Our definition mimic that of \cite[(1.14)]{Miller06}

\remk Since our definition involves a parameter $s$, so it is reasonable to put the parameter in the notation $\cS_\rho(F_v^\times)$.
However we choose not to do so to ease the notation and indicate that we are considering the specific $s$ when necessary.

We introduce our basic function with parameters as follows.
\begin{definition}Let $v$ be a place of $F$, define
\begin{equation}\label{eqn-basic-function-parameter}
\bL_v(u, s)=\zeta_{F_v}(2-2s)\int_{F_v }\psi(ux)t(\nu(x))^{2-2s}dx, 
\end{equation}
which converges to a holomorphic function in $s$ if $\Re(s)<1/2$. In general, it is understood 
to be a meromorphic continuation.
\end{definition}

\begin{lemma}
When $s=1/2$, $v$ non-archimedean and $e_v=0$, our definition of basic function coincides with that of Ngo above up to scalar.
\end{lemma}

\begin{proof}
In fact by \cite[Lemma 6.2]{Wri85}
for $\Re(s)<1/2$, 
\begin{align}\label{eqn-definition-basic-function-Eisenstein}
\int_{F_v }\psi(ux)t(\nu(x))^{2-2s}dx=|\mathfrak{D}_v|_v^{s-1/2}\frac{\sigma_{2s-1}(u\omega_v)}{\zeta_{F_v}(2-2s)}
\end{align}

where $\mathfrak{D}_v=(\omega_v)$ is the different of $F_v$ over $\Q_p$(the discrepancy lies in the definition of the factor $\zeta_{F_v}(s)$) and
\[
\sigma_{s}(u)=\left\{\begin{array}{cc}
\sum_{i=0}^{\nu_v(u)}q_v^{is}, &\text{ if $|u|_v\leq 1$ },\\
0, &\text{ otherwise.}
\end{array}\right. 
\]
Now the result follows.
\end{proof}

We need to know the asymptotic behavior of the basic function near $x=0$ and $x=\infty$.

\begin{prop}\label{lemma-estimate-near-boundary}\noindent 
\begin{itemize}
\item[(1)] For $|x|_v>>1$, we have 
\[
\bL_v(x, s)\sim \left\{ \begin{array}{cc}
&O(|x|^{-s}e^{-2\pi |x|_v}), \quad \text{ if  $F_v=\R$ };\\
&O(|x|^{1/2-2s}e^{-4\pi |x|_v^{1/2}}), \quad \text{ if  $F_v=\C$ };\\
&0, \quad \text{ otherwise;}
\end{array}\right.
\]
\item[(2)]For $|x|_v<<1$, we have 
\begin{align*}
\bL_v(x, s)\sim \left\{ \begin{array}{cc}
O(1), &\quad \text{ if  $v$ is archimedean  and $\Re(s)< 1/2$};\\
O(|x|_v^{1-2s}), &\quad \text{ if  $v$ is archimedean  and $\Re(s)\geq 1/2$ and $s\neq 1/2$};\\
O(\log(|x|_v)), &\quad \text{ if $s=1/2$};\\
O(|x|_v^{2s-1}) , &\quad \text{ if $v$ is non-archimedean $\Re(s)\leq1/2$ and $s\neq 1/2$}.
\end{array}\right.
\end{align*}
\end{itemize}
\end{prop}
\begin{proof}
If $v$ is archimedean, by \cite[\S 6]{Wri85} that we have
\[
\bL_v(x, s)= \left\{ \begin{array}{cc}
&|x|_v^{1/2-s}K_{1/2-s}(2\pi|x|_v), \quad \text{ if  $F_v=\R$ };\\
&|x|_v^{1/2-s}K_{1-2s}(4\pi|x|_v^{1/2}), \quad \text{ if  $F_v=\C$; }
\end{array}\right.
\]
where $K_s(x)$ is the modified Bessel function of the second kind. And the asymptotic behavior follows from that of the Bessel functions.
And for $v$ non-arhimedean, from (\ref{eqn-definition-basic-function-Eisenstein}) we have 
\[
\bL_v(x, s)=|\mathfrak{D}_v|_v^{s-1/2}\sigma_{2s-1}(u\omega_v).
\]
\end{proof}

As a consequence, we have the following
\begin{cor}\label{cor-characterizing-basic-function-behavior}\noindent
\begin{itemize}
\item[(1)] The basic function $\bL_v$ is rapidly decaying when $|x|_{v}\rightarrow \infty$.
\item[(2)] Assume that $v$ is non-archimedean, then space of functions $U_v=\langle\bL(a x, s): a\in F_v^\times \rangle$ is finite dimensional for $|x|_v<<1$.
\end{itemize}
\end{cor}

\begin{lemma}
Let $v$ be non-archimedean. For $w\in \cO_{v}^\times$, we have 
\[
w.\bL_v(u, s)=\bL_v(u,s).
\]
\end{lemma}
\begin{proof}
Clear from the definition.
\end{proof}

Keep the assumption that $v$ is  non-archimedean. We can define an action of $\C[\epsilon]$ on $f(u,s)$ by letting
\[
\epsilon(f)(u,s)=f(\pi_v^{-1}u, s).
\]

Now it follows from Proposition \ref{prop-polynomial-basic-compactly-supported} that there exists $P\in \C[\epsilon]$ such that
\[
P(\bL_v)\in \cC_c^\infty(F_v^\times).
\]
In general we define
\begin{definition}
Suppose that $v$ is non-archimedean and $f\in \cS_\rho(F_v^\times)$. Assume that $P_{f}(\epsilon)$ is the minimal non-trivial polynomial satisfying 
\[
P(f)\in \cC_c^\infty(F_v^\times).
\]
\end{definition}

The polynomial $P_{f}(\epsilon)$ is the analogue of Bernstein polynomials
in the theory of $D$-modules. We further deduce that 

\begin{prop}\label{prop-structure-rho-schwartz}
Let $v$ be non-archimedean.
The space $\cS_{\rho}(F_v^\times)$ is finite dimensional at $x=0$, by which we mean that for $f\in \cS_{\rho}(F_v^\times)$, 
\[
f=f_0+\sum_{i=1}^r c_i P_i(\bL_v), \quad c_i\in \C, f_0\in \cC_c^\infty(F_v^\times),
\]
where $\{P_1, \cdots, P_r\}$ is a set of representatives for $\C[\epsilon]/(P_{\bL_v})$.
\end{prop}

\begin{proof}
It follows from the existence of $P_{\bL_v}$.
\end{proof}
\begin{remark} Of course, the statement fails for archimedean places, but the function $\bL_{v}(u, s)$ such that $F_v=\R$ does satisfy a second order differential equation:
for $\partial=u\frac{d}{du}$, 
\begin{equation}\label{eqn-differential-equation-basic-R}
\partial^2\bL_v+(2s-1)\partial\bL_v-4\pi^2 u^2\bL_v=0.
\end{equation}
Similarly for $F_v=\C$ and $\partial=z\frac{d}{dz}$, 
\begin{equation}\label{eqn-differential-equation-basic-C1}
\partial^2\bL_v+(2s-1)\partial \bL_v-4\pi^2 |z|_v\bL_v=0,
\end{equation}
and 
\begin{equation}\label{eqn-differential-equation-basic-C2}
\partial\bL_v=\bar{\partial}\bL_v.
\end{equation}
\end{remark}

\begin{lemma}\label{lem-local-structure-rho-Schwartz}
We have
\[
\bL_v(x,s)\in \cS_\rho(F_v^\times).
\]
Moreover, for $v$ non-archimedean and $0<\delta<<1$, the space 
$\{f|_{\{x: |x|_v<\delta\}}: f\in \cS_\rho(F_v^\times)\}$ coincides with the space $U_v$ in 
Corollary \ref{cor-characterizing-basic-function-behavior}.
\end{lemma}
\begin{proof}
It follows immediately from Proposition \ref{lemma-estimate-near-boundary} and its corollary.
\end{proof}

We also need to introduce a $\rho$-Bessel function with parameter $s$ as follows.
\begin{definition}For $v$ a place of $F$, define
\begin{equation}\label{eqn-definition-rho-bessel-function}
K_{F_v}(x, s)=\int_{F_v^\times}|w|_{v}^{1-2s}\psi_v(-w-w^{-1}x)d^\times w.
\end{equation}
\end{definition}

Let us give another interpretation of the function $K_{F_v}(x, s)$ in terms of gamma factors. 
Assume that $v$ is archimedean and $F_v=\R$, then the gamma factor is given by
\[
\gamma_{v}(s, \rho)=\frac{L_{v}(1-s, \triv, \rho)}{L_{v}(s, \triv, \rho)}=(\frac{L_{v}(1-s, \triv)}{L_{v}(s, \triv)})^2, \quad L_{v}(s, \triv)=\pi^{-s/2}\Gamma(s/2).
\]
And we have
\[
K_{\R}(x, 1/2)=\frac{1}{2\pi i}\int_{\Re(s)=\sigma}\gamma_{v}(2s-1, \rho)^{-1}x^{1-2s} ds.
\]
with $\sigma>1$.

Note that when $v$ is archimedean,  the integral
\[
 \int_{F_v^\times }|w|_v^{1-2s}\psi_v(-w-w^{-1}x)d^\times w
\]
is conditionally convergent for $0<\Re(s)<1/2$, hence for $s=1/2$ we have to 
understand it in distributional sense. Also, for $v$ non-arhimedean, as in \cite{Sal66}, it is understood that the integral 
on the right hand side of (\ref{eqn-definition-rho-bessel-function}) is taking to be its Principal values:
\[
\lim_{n\rightarrow\infty}\int_{q_v^{-n}<|x|_v<q_v^{n}}|w|_{v}^{1-2s}\psi_v(-w-w^{-1}x)d^\times w.
\]

In particular for $f\in \cS(F_v^\times)$, we define 
\begin{definition}We define
\[
\cH_s(f)(x, s)=\int_{F_v^\times }|y|^{2s}f(y)K_{F_v}(xy, s)d^\times y=\int_{F_v^\times }\int_{F_v^\times}f(y)|w|^{1-2s}\psi(-wy-w^{-1}x)d^\times wd y.
\]
\end{definition}

\remk The above definition coincides with the Hankel transform in \cite[\S6]{Sal66} if $\Re(s)=1/2$.   
Also it is related to the Hankel transform in \cite[(16)]{Bump06}. Taking $n=1$ and replacing $s$ by $1-s$ in loc.cit., the kernel function 
 \cite[(9)]{Bump06} becomes $|x|_v^{s}K_{F_v}(x,s)$, if we denote by $H_s$ the Hankel transform in loc.cit., we obtain
 \[
 H_s(\phi)(x)=|x|_v^s\int_{F_v}\phi(y)|y|_v^{s-1/2}K_{F_v}(xy) dy=|x|_v^s\cH_s(|y|_v^{1/2-s}\phi).
 \]

\begin{prop}
The functional
\[
\phi\in \cS(F_v)\mapsto \langle\phi, y^{2s-1}K_{F_v}(xy)  \rangle: =\int_{F_v^\times }|y|^{2s-1}\phi(y)K_{F_v}(xy, s)dy
\]
defines a distribution on $\cS(F_v)$.
\end{prop}
\begin{proof}
For $v$ non-archimedean, it follows from that $y^{2s-1}K_{F_v}(xy)$ is locally integrable. The latter follows from the computation in 
\cite[Theorem 8]{Sal66}. For $v$ archimedean, it follows from \cite[Proposition 3]{Bump06}.
\end{proof}

\begin{prop}\label{lem-Hankel-Fourier}
Assume that $\Re(s)\leq 1/2$.
For $\varphi\in \cS(F_v)$ be a Schwartz function, then 
\begin{equation}\label{eqn-new-definition-Hankel-transform}
\cH_s \varphi(u)=(|x|_v^{2s-2}\hat{\varphi}(1/x))^{\hat{}}(u)
\end{equation}
where $\hat{ }$ denotes the Fourier transform with respect to $\psi_v$.
\end{prop}

\begin{proof}
Note that our definition of Hankel transform coincide with the one in \cite[\S 6]{Sal66}
if $v$ is non-archimedean(although in loc.cit it is assumed that $\Re(s)=1/2$ but the proof will not change). And in this case the lemma follows from 
Lemma 10 of loc.cit.. The proof for archimedean place follows from the proof of  Proposition 4 of \cite{Bump06}.
\end{proof}
\remk It follows from the Proposition that if $\Re(s)=1/2$, we have
\[
||\cH_s(f)||_{L^2(F_v, dx)}=||f||_{L^2(F_v, dx)}.
\]
\\

Let us recall the notion of Sobolev spaces, which are defined to be 
\begin{align*}
H^{\mu}(F_v)&=\{f\in L^2(F_v, dx): (1+|x|_v^2)^{\Re(\mu)/2}\hat{f}(x)\in L^2(F_v, dx)\}, \quad \Re(\mu)\geq 0,\\
H^{\mu}(F_v)&=\{f\in \cS(F_v)': (1+|x|_v^2)^{\Re(\mu)/2}\hat{f}(x)\in L^2(F_v, dx)\}, \quad \Re(\mu)< 0.
\end{align*}
And define the norm for $ H^{\mu}(F_v)$ by
\[
||f||_{H^{\mu}(F_v)}=||(1+|x|_v^2)^{\Re(\mu)/2}\hat{f}||_{L^2(F_v, dx)}.
\]
Note that $H^{\mu}(F_v)$ is dual to $H^{-\mu}(F_v)$.

\remk In more classical situation when 
$v$ is archimedean and $\mu\in \Z_{+}$, we have the following equivalent definition for $H^{\mu}(F_v)$:
\[
H^{\mu}(F_v)=\{f\in L^2(F_v, dx): (\frac{d}{dx})^j f\in L^2(F_v, dx), j=0,\cdots, \mu\}.
\]
\\

\begin{prop}\label{prop-embedding-sobolev}
For $f\in \cS(F_v)$, we have 
\[
||\cH_s f||_{H^{1-2s}(F_v)}= ||f||_{H^{1-2s}(F_v)}.
\]
As a consequence, the transform defined by the right hand side of (\ref{eqn-new-definition-Hankel-transform})
extends to an isometry of 
the Sobolev space $H^{1-2s}(F_v)$ for any $s\in \C$ with $\Re(s)\leq 1/2$. And this will be our general definition of Hankel 
transform.
\end{prop}

\begin{proof}
We will use Proposition \ref{lem-Hankel-Fourier} to argue. 
Let $f\in \cS(F_v)$. Then
\[
\cH_s f(u)=(|x|_v^{2s-2}\hat{f}(1/x))^{\hat{}}(u).
\]
Then
\[
||\cH_s f||_{H^{1-2s}(F_v)}=|| (1+|x|_v^2)^{1/2-\Re(s)}|x|_v^{2s-2}\hat{f}(1/x)||_{L^2(F_v, dx)}
\]
Now we have 
\begin{align*}
&|| (1+|x|_v^2)^{1/2-\Re(s)}|x|_v^{2s-2}\hat{f}(1/x)||_{L^2(F_v, dx)}^2\\
&=\int_{F_v} (1+|x|_v^2)^{1-2\Re(s)}|x|_v^{4\Re(s)-4}|\hat{f}(1/x)|^2 dx\\
&=\int_{F_v} (1+|x|_v^{-2})^{1-2\Re(s)}|x|_v^{-4\Re(s)+2}|\hat{f}(x)|^2 dx\\
&=\int_{F_v} (1+|x|_v^{2})^{1-2\Re(s)}|\hat{f}(x)|^2 dx\\
&=||f||_{H^{1-2s}(F_v)}^2.
\end{align*}
\end{proof}

\begin{remark}
By duality, we can further extend the Hankel transform via
\[
\langle \cH_{1-s}(f), \phi\rangle=\langle f, \cH_s(\phi)\rangle, \quad \Re(s)\leq 1/2
\]
for $f\in H^{2s-1}(F_v)$ and $\phi\in H^{1-2s}(F_v)$. In this case we still have 
\[
\cH_{1-s} f(u)=(|x|_v^{-2s}\hat{f}(1/x))^{\hat{}}(u).
\]
\end{remark}

\begin{cor}\label{cor-extension-definition-Hankel-transform}
Assume that $\Re(s)\leq 1/2$. Then  $\cS_\rho(F_v^\times)$ is a subspace of $H^{\delta}(F_v)$ for $0\leq \delta<3/2-2\Re(s)$. In particular, 
we can extend the definition of Hankel transform to $\cS_\rho(F_v^\times)$ by applying (\ref{eqn-new-definition-Hankel-transform}).
\end{cor}

\begin{proof}
We note that for $\Re(s)\leq 1/2$, $\cS_\rho(F_v)$ is a subspace of $L^2(F_v, dx)$.  
We first consider the case where $1-2s\neq 0$. Note that we only need to show that the space $ |u|_{v}^{1-2s}\cS(F_v)$ is contained
in $H^{1-2s}(F_v)$. Let $\sigma\in |u|_{v}^{1-2s}\cS(F_v)$, we write
\[
\sigma= |u|_{v}^{1-2s}f(u), \quad f\in \cS(F_v).
\]
And
\[
\hat{\sigma}(x)=\int_{F_v}|u|_{v}^{1-2s}f(u)\psi_v(xu) du.
\]
First we note that  for $v$ non-archimedean, $\hat{\sigma}(x)$ is of compact support, hence 
\[
||\cH_s \sigma||_{H^{\delta}(F_v)}=|| (1+|x|_v^2)^{\delta}\hat{\sigma}(x)||_{L^2(F_v, dx)}<\infty
\]
for any $\delta>0$.
It remains to consider the case when $v$ is archimedean. Since 
\[
\frac{d^n}{dx^n}\hat{\sigma}(x)=C\int_{F_v}|u|_{v}^{1-2s}(u^nf(u))\psi_v(xu) du, \quad C\in \C^\times.
\]
and $u^nf(u)$ is a Schwartz function, the above integration is always well defined. 
Therefore $\hat{\sigma}(x)$ is in the class $\cC^\infty(F_v)$. 
We only need to know its asymptotic behavior when $|x|_v\rightarrow \infty$.
Let us show by using dyadic decomposition that $\hat{\sigma}(x)$ is of decay
$|x|_v^{2\Re(s)-2}$ as $|x|_v\rightarrow \infty$. In fact let us consider the 
decomposition 
\[
F_v=\coprod_{k=-1}^{\infty} \Delta_k, \quad \Delta_{-1}=\{|x|_v\geq 1\}, \Delta_{k}=\{x: \frac{1}{2^{k}}< |x|_v\leq  \frac{1}{2^{k-1}}\}(k\geq 0).
\]
Let us chose smooth functions $\{\eta_k: k\geq -1\}$ on $F_v$ satisfying 
\[
\supp(\eta_i)\subseteq \Delta_{k}\cup \Delta_{k+1}(k\geq 0), \supp(\eta_{-1})\subseteq \Delta_{-1}
\]
and 
\[
\sum_{k=-1}^{\infty}\eta_k=1, \text{ and } \eta_k\geq 0.
\]
Then
\[
\hat{\sigma}(x)=\sum_{k=-1}^{\infty}\int_{F_v}|u|_{v}^{1-2s}\eta_k(u)f(u)\psi_v(xu) du
\]
Note that $|u|^{1-2s}\eta_{-1}f$ is a Schwartz function, so is its Fourier transform, which is of rapid decay at infinity.
Therefore we can drop the term $k=-1$ and only consider $k\geq 0$. For simplicity, let us assume that 
$F_v=\R$, but the case $F_v=\C$ is similar. We have 
\[
\int_{F_v}|u|_{v}^{1-2s}\eta_k(u)f(u)\psi_v(xu) du= C_\ell |x|_v^{-\ell}\int_{F_v}\psi_v(xu) \frac{d^\ell}{du^\ell}(|u|_{v}^{1-2s}\eta_k(u)f(u))du
\]
for $C_\ell\in \C$ and $\ell\geq 0$. The derivative here can be written as
\[
\sum_{m_1+m_2+m_3=\ell}C_{m_1, m_2, m_3}|u|^{1-2s-m_1} \eta_{k}^{(m_2)}(u)\phi^{(m_3)}(u).
\]
Recall that the support of $ \eta_{k}$ is $ \Delta_{k}\cup \Delta_{k+1}$, 
hence $|u|^{1-2s-m_1}$ is bounded by $2^{k(m_1+2\Re(s)-1)}$.
We further bound 
the derivatives $\eta_k$ and $\phi$ by additional constants. In all we obtain a bound for the derivative term
\[
|\sum_{m_1+m_2+m_3=\ell}C_{m_1, m_2, m_3} 2^{k(m_1+2\Re(s)-1)}|\leq D_{\ell, s}2^{k(\ell+2\Re(s)-1)}
\]
Now the integral term is bounded by
\begin{align*}
&|\int_{F_v}|u|_{v}^{1-2s}\eta_k(u)f(u)\psi_v(xu) du|\\
&\leq D_{\ell, s} |x|_v^{-\ell}\int_{\Delta_{k}\cup \Delta_{k+1}}2^{k(\ell+2\Re(s)-1)}du\\
&= |x|_v^{-\ell}D_{\ell, s}2^{k(\ell+2\Re(s)-2)}.
\end{align*}
Note that although the above bound holds for any $\ell\geq 0$, but for $\ell=0$ we get a bound
of order $2^{-k(2-2\Re(s))}$. We see that taking derivatives only provides improvement on bounds 
if $|x|_v\geq 2^k$.
Therefore
\begin{align*}
&|\sum_{k=0}^{\infty}\int_{F_v}|u|_{v}^{1-2s}\eta_k(u)f(u)\psi_v(xu) du|\\
&\leq \sum_{|x|_v\geq 2^k}D_{\ell, s}|x|_v^{-\ell}2^{k(\ell+2\Re(s)-2)}+\sum_{|x|_v\leq 2^k}D' 2^{k(2\Re(s)-2)}\\
&\leq D_{\ell, s}|x|_v^{-\ell}|x|_v^{\ell+2\Re(s)-2}+D'|x|_v^{2\Re(s)-2}\\
&=D_s |x|_v^{2\Re(s)-2}.
\end{align*}
Finally we have 
\begin{align*}
&||\cH_s \sigma||^2_{H^{\delta}(F_v)}\\
&=|| (1+|x|_v^2)^{\delta/2}\hat{\sigma}(x)||_{L^2(F_v, dx)}^2 \\
&=\int_{F_v}(1+|x|_v^2)^{\delta}|\hat{\sigma}(x)|^2dx\\
&=|\int_{|x|_v\leq 1}(1+|x|_v^2)^{\delta}|\hat{\sigma}(x)|^2dx|+|\int_{|x|_v\geq 1}(1+|x|_v^2)^{\delta}|\hat{\sigma}(x)|^2dx|\\
&\leq |\int_{|x|_v\leq 1}(1+|x|_v^2)^{\delta}|\hat{\sigma}(x)|^2dx|+D_s \int_{|x|_v\geq 1}|x|_v^{2\delta} |x|_v^{4\Re(s)-4}dx\\
&=\int_{|x|_v\leq 1}(1+|x|_v^2)^{\delta}|\hat{\sigma}(x)|^2dx+D_s\int_{|x|_v\geq 1}|x|_v^{2\delta+4\Re(s)-4} dx.
\end{align*}
The condition $\delta<3/2-2\Re(s)$ implies that $2\delta+4\Re(s)-4<-1$, hence the above integral
converges absolutely.
This finishes the proof of the fact that $||\cH_s \sigma||_{H^{1-2s}(F_v)}<\infty$.

Next consider the case when $s=1/2$. We repeat the above computation with $s=1/2$
and  obtain
\begin{align*}
&|\int_{F_v}|\log(|u|_v)|\eta_k(u)f(u)\psi_v(xu) du|\\
&= |x|_v^{-\ell}D_{\ell, s}2^{k(\ell-1)},
\end{align*}
for $\ell\geq 1$. For $\ell=0$ we obtain a bound of order $2^{-k}k$. And as before, 
\begin{align*}
&|\sum_{k=0}^{\infty}\int_{F_v}|\log(|u|_v)|\eta_k(u)f(u)\psi_v(xu) du|\\
&\leq \sum_{|x|_v\geq 2^k}D_{\ell, s}|x|_v^{-\ell}2^{k(\ell-1)}+\sum_{|x|_v\leq 2^k}D' 2^{-k}k\\
&\leq D_{\ell, s}|x|_v^{-\ell}|x|_v^{\ell-1}+D'|x|_v^{-1}|\log(|x|_v)|\\
&\leq D'|x|_v^{-1+\epsilon}.
\end{align*}
for any $\epsilon>0$.
Now for $0\leq \delta<1/2$, we have 
\begin{align*}
&||\cH_s \sigma||^2_{H^{\delta}(F_v)}\\
&=|| (1+|x|_v^2)^{\delta/2}\hat{\sigma}(x)||_{L^2(F_v, dx)}^2 \\
&=\int_{F_v}(1+|x|_v^2)^{\delta}|\hat{\sigma}(x)|^2dx\\
&=|\int_{|x|_v\leq 1}(1+|x|_v^2)^{\delta}|\hat{\sigma}(x)|^2dx|+|\int_{|x|_v\geq 1}(1+|x|_v^2)^{\delta}|\hat{\sigma}(x)|^2dx|\\
&\leq |\int_{|x|_v\leq 1}(1+|x|_v^2)^{\delta}|\hat{\sigma}(x)|^2dx|+D_s \int_{|x|_v\geq 1}|x|_v^{2\delta} |x|_v^{-2+2\epsilon}dx\\
&=\int_{|x|_v\leq 1}(1+|x|_v^2)^{\delta}|\hat{\sigma}(x)|^2dx+D_s\int_{|x|_v\geq 1}|x|_v^{2\delta-2+2\epsilon} dx.
\end{align*}
This finishes the proof of the fact that $||\cH_s \sigma||_{H^{\delta}(F_v)}<\infty$.

\end{proof}

Let us recording the following fact.
\begin{cor}\label{cor-decaying-schwartz-function}
Let $f\in \cS_\rho(F_v^\times)$, then 
$\hat{f}\in \cC^{\infty}(F_v)$ decays of order $|x|_v^{2\Re(s)-2+\epsilon}$ as 
$|x|_v\rightarrow\infty$ for any $\epsilon >0$.
\end{cor}
\remk The constant $\epsilon$ can be dropped if $s\neq 1/2$.
Also note that in general for $f\in L^2(F_v, dx)$ with compact support, 
it is not true that $\hat{f}$ is of compact support even if $v$ is non-archimedean.

\begin{lemma}\label{lemma-eigenfunction-Hankel-transform}
Assume that $e_v=1$ if $v$ is non-archimedean. Then we have 
\[
\cH_s(\bL_v)(x, s)=\bL_v(x, s).
\]
\end{lemma}

\begin{proof}
Recall that by definition, 
\[
\hat{\bL}_v(z,s)=\zeta_{F_v}(2-2s)t(\nu(z))^{2-2s}.
\]
By Corollary \ref{cor-extension-definition-Hankel-transform}, 
\begin{align*}
\cH_s(\bL_v)=\zeta_{F_v}(2-2s)(|z|_v^{2s-2}t(\nu(z^{-1}))^{2-2s})^{\hat{}}
\end{align*}

For $v$ non-archimedean, observe that 
\begin{align*}
t(\nu(z))=\left\{\begin{array}{cc}
|z|_v^{-1}, &\text{ if } |z|_v> 1; \\
1,& \text{ otherwise.}
\end{array}\right.
\end{align*}
Therefore 
\begin{align*}
\int_{F_v}t(\nu(z^{-1}))^{2-2s}|z|_v^{2s-2}\psi_v(zx)dz&=\int_{|z|_v\leq 1}\psi_v(zx)dz+\int_{|z|_v> 1}|z|_v^{2s-2}\psi_v(zx)dz\\
\end{align*}
where the latter is equal to $\int_{F_v}t(\nu(z))^{2-2s}\psi_v(zx)dz$. 

Now let $v$ be such that $F_v=\R$ we have $t(\nu(z))=(1+|z|_v^2)^{-1/2}$ while for $F_v=\C$ we have $t(\nu(z))=(1+|z|_v)^{-1}$. Hence if $F_v=\R$, 
\begin{align*}
 \int_{F_v}t(\nu(z^{-1}))^{2-2s}|z|_v^{2s-2}\psi_v(zx)dz&= \int_{F_v}(1+|z|_v^{-2})^{s-1}|z|_v^{2s-2}\psi_v(zx)dz\\
 &= \int_{F_v}(1+|z|_v^{2})^{s-1}\psi_v(zx)dz\\
 &=\int_{F_v}t(\nu(z))^{2-2s}\psi_v(zx)dz.
\end{align*}
If $F_v=\C$, we have 
\begin{align*}
 \int_{F_v}t(\nu(z^{-1}))^{2-2s}|z|_v^{2s-2}\psi_v(zx)dz&= \int_{F_v}(1+|z|_v^{-1})^{2s-2}|z|_v^{2s-2}\psi_v(zx)dz\\
 &= \int_{F_v}(1+|z|_v)^{2s-2}\psi_v(zx)dz\\
 &=\int_{F_v}t(\nu(z))^{2-2s}\psi_v(zx)dz.
\end{align*}
\end{proof}

\begin{prop}\label{prop-schwartz-stable-p-adic}
Suppose that $v$ is non-archimedean. Let $f\in \cS_\rho(F_v^\times)$, then $\cH_s(f)\in  \cS_\rho(F_v^\times)$.
\end{prop}

\begin{proof}
In view of Lemma \ref{lem-local-structure-rho-Schwartz}, 
it remains to show the Proposition for $f\in \cC_c^\infty(F_v^\times)$. Assume that $f\in \cC_c^\infty(F_v^\times)$. 
It follows from \cite[Theorem 13]{Sal66} that for $|x|>>1$, we have 
\[
\cH_s(f)(x, s)=0.
\]
In fact by definition, 
\begin{align*}
\cH_s(f)(x,s)&=\int_{F_v^\times }|y|_v^{2s}f(y)K_{F_v}(xy, s)d^\times y\\
&=\int_{F_v^\times }\int_{F_v^\times}f(y)|w|_v^{1-2s}|y|_v^{2s}\psi(-w-w^{-1}xy)d^\times wd^\times y\\
&=\int_{F_v}\int_{F_v^\times}f(y)|w|_v^{1-2s}\psi(-wy-w^{-1}x)d^\times wd y, 
\end{align*}
which is exactly the Hankel transform studied in  \cite[\S 5]{Sal66}. Next we study the behavior of $\cH_s(f)(x, s)$ near $x=0$.
Since $ \cC_c^\infty(F_v^\times)$ is generated by the set of characteristic functions $1_{a+\pi_v^r\Z_p}( a\in \cO_v \text{ but }a\not\equiv 0 \mod \pi_v^r)$, 
we can restrict ourselves to consider $f=1_{a+\pi_v^r\cO_v}$ with $a\not\equiv 0 \mod \pi_v^r$. Moreover, replacing
$f$ by $a^{-1}.f$ we can further reduce to the case where  $f=1_{1+\pi_v^r\cO_v}$
Therefore we have
\begin{align*}
\cH_s(f)(x, s)=\int_{1+\pi_v^r\cO_v}\int_{F_v^\times}|w|_v^{1-2s}\psi_v(-w-w^{-1}xy)d^\times wd^\times y
\end{align*}
Note that it suffices to show that for $|x|_v<<1$, we have 
\[
\cH_s(f)(x, s)=c(s) \bL_v(x, s), \quad c(s)\in \C(q_v^{-s}).
\]
In fact, by \cite[Theorem 8]{Sal66}, if $|xy|=|x|<|\omega_v|^2$(note the difference on the conductor of the additive character chosen), we have
\[
\int_{F_v^\times}|w|_v^{1-2s}\psi(-w-w^{-1}xy)d^\times w=|x\omega_v^{-1}|_v^{1-2s}\gamma_v(2-2s)+\gamma_{v}(2s)|\omega_v|_v^{1-2s}.
\]

Hence 
\[
\cH_s(f)(x, s)=\vol(1+\pi_v^r\cO_v, d^\times y)(|x\omega_v^{-1}|_v^{1-2s}\gamma_v(2-2s)+\gamma_{v}(2s)|\omega_v|_v^{1-2s}).
\]
On the other hand,  
it follows from the formula (\ref{eqn-definition-basic-function-Eisenstein}) that for $|x|_v<1$, 
\begin{align*}
\int_{F_v}t(\nu(z))^{2-2s}\psi(zx)dz&=\frac{1-|x|_v^{1-2s}q_v^{2s-1}}{\zeta_{F_v}(2-2s)(1-q_v^{2s-1})}\\
&=-|\pi_F|_v^{2s-1}\gamma_{v}(2-2s)+|x|_v^{1-2s}\gamma_{v}(2-2s).
\end{align*}
Now it is clear that the function 
\[
\cH_s(f)-\frac{\vol(1+\pi_v^r\cO_v, d^\times y)|\omega_v|_v^{2s-1}}{\zeta_{F_v}(2-2s)}\bL_{v}.
\]
is constant around $x=0$. It remains to show that such functions are in $\cS_\rho(F_v^\times)$.
Finally, we observe that  
\[
|\pi_v|^{2s-1}\bL_v-\pi_v.\bL_v
\]
is constant around $x=0$.

\end{proof}
\remk One can also apply the Corollary of \cite[Lemma 10]{Sal66} to conclude.
\\

\begin{lemma}
For any character $\chi: F_v^\times \rightarrow\C^\times$, in $\cS(F_v)'$
we have
\begin{equation}\label{eqn-computation-orbital-integral-GL_1}
\int_{F_v^\times }\chi(u)|u|_v^{s}\psi_v(ux)d^\times u=\chi(x)^{-1}|x|_v^{-s}\gamma_v(-s+1, \chi^{-1})
\end{equation}
where
\[
\gamma_v(s, \chi)=\epsilon(s, \chi)\frac{L_{v}(1-s, \chi^{-1})}{L_{v}(s, \chi)}.
\]
In particular, when $\chi=1$ is the trivial character,
\[
\gamma_v(s)=\gamma_v(s, 1)=\frac{\zeta_{F_v}(1-s)}{\zeta_{F_v}(s)}.
\]

\end{lemma}
\begin{proof}
This follows from the local functional equation for the character $\chi$, cf. \cite[Corollary 3.7 and the Remark follows]{Kud04}, see \cite{Deng22} for a related discussion.
\end{proof}

\begin{prop}\label{prop-local-functional-equation-tame-schwartz}
Suppose that $\Re(s)\leq 1/2$. Let $\phi\in \cS_\rho(F_v^\times)$, then we have for $2\Re(s)-1< \Re(\mu)< 1/2$,
\[
\cM(\cH_s(\phi))(\chi, \mu)=\gamma_v(1-\mu, \chi^{-1})\gamma_v(2s-\mu, \chi^{-1})\cM(\phi)(2s-\mu).
\]
\end{prop}
\begin{proof}
Let us determine the Mellin transform $\cM(\cH_s(\phi))$ for a compactly supported function $\phi\in  \cC_c^\infty(F_v^\times)$. In fact we have for $\Re(\mu)>0$,
\begin{align}
&\cM(\cH_s(\phi))(\chi, \mu)\nonumber\\
&=\int_{F_v\times} \cH_s(\phi)(x)\chi(x)|x|_v^{\mu} d^\times x\nonumber\\
&=\langle (|u|_v^{2s-2}\hat{\phi}(1/u))^{\hat{}}, \chi(x)|x|_v^{\mu-1} \rangle\label{eqn-step-mellin-schwartz}\\
&=\langle (|u|_v^{2s-2}\hat{\phi}(1/u)), \hat{\chi(x)|x|_v^{\mu-1}}(u) \rangle\nonumber\\
&=\langle (|u|_v^{2s-2}\hat{\phi}(1/u)),\chi(u)^{-1}|u|_v^{-\mu}\gamma_v(1-\mu, \chi^{-1}) \rangle\label{eqn-plug-in-gamma-factor1}\\
&=\langle \hat{\phi}(1/u),\chi(u)^{-1}|u|_v^{2s-2-\mu}\gamma_v(1-\mu, \chi^{-1}) \rangle\nonumber\\
&=\gamma_v(1-\mu, \chi^{-1})\langle \hat{\phi}(u),\chi(u)|u|_v^{-2s+\mu} \rangle\nonumber\\
&=\gamma_v(1-\mu, \chi^{-1})\langle \phi(x),(\chi(u)|u|_v^{-2s+\mu})^{\hat{}} \rangle\nonumber\\
&=\gamma_v(1-\mu, \chi^{-1})\gamma_v(2s-\mu, \chi^{-1})\langle \phi(x), \chi(u)^{-1}|u|_v^{2s-\mu-1}\rangle\label{eqn-plug-in-gamma-factor2}\\
&=\gamma_v(1-\mu, \chi^{-1})\gamma_v(2s-\mu, \chi^{-1})\cM(\phi)(2s-\mu)
\end{align}
where we use the fact that $\chi(x)|x|_v^{\mu}\in \cS(F_v)'(\Re(\mu)> 0)$ in (\ref{eqn-step-mellin-schwartz}) and 
apply (\ref{eqn-computation-orbital-integral-GL_1}) in (\ref{eqn-plug-in-gamma-factor1}) and (\ref{eqn-plug-in-gamma-factor2}).
We want to extend this to $H^{\delta}(F_v)(0\leq \delta<3/2-2\Re(s))$ by continuity
since by Corollary \ref{cor-extension-definition-Hankel-transform}, the latter contains $\cS_\rho(F_v^\times)$. 
In fact by (\ref{eqn-computation-orbital-integral-GL_1})
\[
\hat{\chi(x)|x|_v^{\mu-1}}=\chi(u)^{-1}|u|_v^{-\mu}\gamma_v(-\mu+1, \chi^{-1}).
\]
Therefore
\begin{align*}
||\chi(x)|x|_v^{\mu-1}||_{H^{-\delta}(F_v)}^2&=\int_{F_v}(1+|u|_v^2)^{-\delta}|\chi(u)^{-1}|u|_v^{-\mu}\gamma_v(-\mu+1, \chi^{-1})|^2 du\\
&=|\gamma_v(-\mu+1, \chi^{-1})|^2\int_{F_v}(1+|u|_v^2)^{-\delta}|u|_v^{-2\Re(\mu)} du.
\end{align*}
Note that for $1/2-\delta< \Re(\mu)< 1/2$, the above integral converge absolutely. 
Hence  $\chi(x)|x|_v^{\mu-1}\in H^{-\delta}(F_v)=H^{\delta}(F_v)'$, the same argument as
in the compactly supported case yields the proposition.

\end{proof}

\begin{cor}\label{lem-mellin-hankel-compact-support}
For $\phi\in \cS_\rho(F_v^\times)$, we have 
\[
\cH_s(\phi)(x)=\sum_{\chi\in \pi_0(\hat{F}_v^\times)}\int_{\R}a_v(\chi, \sigma+i t)\cM\phi(\chi^{-1}, 2s-\sigma-i t) \chi(x)dt
\]
where $\sigma>>0$ and 
\[
a_v(\chi, \mu)=\gamma_v(2s-\mu, \chi^{-1})\gamma_v(1-\mu, \chi^{-1}).
\]
\end{cor}
\begin{proof}
Apply the inversion formula to conclude.
\end{proof}

\remk The lemma is proved in more general set up in \cite[Lemma 5.2]{Ichi13}

\begin{definition}
Following Godment and Jacquet\cite[\S 8]{God06}, we introduce $\cS_\rho^\circ(F_v^\times)$ to be 
the space of functions 
\[
P(u, \partial)\bL_v(u,s), P\in \C[u, \partial], \quad { if  }  F_v=\R, a\in F_v^\times
\]
and
\[
P(z, \bar{z}, \partial)\bL_v(z,s), P\in \C[z, \bar{z}, \partial], \quad { if  } F_v=\C, a\in F_v^\times.
\]
\end{definition}

\begin{remark} \label{remk-generator-tame-Schartz}
In case $F_v=\C$ we do not include $\bar{\partial}$ because of (\ref{eqn-differential-equation-basic-C2}). Furthermore, by 
(\ref{eqn-differential-equation-basic-C1}) and (\ref{eqn-differential-equation-basic-R}), any element of $\cS_\rho^\circ(F_v^\times)$
is of the form
\[
P_1\bL_v+P_2\partial \bL_v
\]
where $P_i(i=1,2)$ is in $\C[u]$(resp. $\C[z, \bar{z}]$) if $F_v=\R$(resp. $F_v=\C$).
\end{remark}

When $\chi=1$ is the trivial representation, we will keep writing $\cM(f)(\mu)$ instead of $\cM(f)(1,\mu)$ .

\begin{lemma}\label{cor-Hankel-commutes-with-derivative}
Assume that $v$ is archimedean. We have 
\[
\cH_s (\partial\varphi)(u)=-2s\cH_s(\varphi)-\partial(\cH_s(\varphi)).
\]
\end{lemma}

\begin{proof}
We consider the case $F_v=\R$ and the case $F_v=\C$ is similar.
Applying Proposition \ref{lem-Hankel-Fourier}, 
we have for $\partial=x\frac{d}{dx}$, 
\begin{align*}
\cH_s (\partial\varphi)(u)&=(|x|_v^{2s-2}\widehat{\partial\varphi}(1/x))^{\hat{}}(u)\\
&=(|x|_v^{2s-2}(-\hat{\varphi}-\partial \hat{\varphi})(1/x))^{\hat{}}(u)\\
&=-\cH_s(\varphi)-(|x|_v^{2s-3}\frac{d\hat{\varphi}}{dx}(1/x))^{\hat{}}(u)
\end{align*}

Note that 
\[
\frac{d}{dx}(|x|_v^{2s-1}\hat{\varphi}(1/x))=(2s-1)|x|^{2s-2}\hat{\varphi}(1/x)-|x|_v^{2s-3}(\frac{d\hat{\varphi}}{dx})(1/x), 
\]
and
\[
x\frac{d}{dx}(|x|_v^{2s-2}\hat{\varphi}(1/x))=\frac{d}{dx}(|x|_v^{2s-1}\hat{\varphi}(1/x))-|x|_v^{2s-2}\hat{\varphi}(1/x).
\]
From the above two equality, we obtain
\[
\partial(|x|_v^{2s-2}\hat{\varphi}(1/x))=(2s-2)|x|_v^{2s-1}\hat{\varphi}(1/x)-|x|_v^{2s-3}(\frac{d\hat{\varphi}}{dx})(1/x)
\]
Hence 
\[
\cH_s (\partial\varphi)(u)=-(2s-1)\cH_s(\varphi)(u)+[\partial(|x|_v^{2s-2}\hat{\varphi}(1/x))]^{\hat{}}(u)
\]
Since
\[
\hat{\partial f}=-\hat{f}-\partial\hat{f},
\]
we have
\[
\cH_s (\partial\varphi)(u)=-2s\cH_s(\varphi)-\partial(\cH_s(\varphi)).
\]

\end{proof}

\begin{prop}\label{prop-stability-tame-Schwartz}
The space $\cS_\rho^\circ(F_v^\times)$ is stable under Hankel transform.
\end{prop}

Although we may not need it, we make a detail computation of the the Fourier-Mellin transform of elements
of $\cS_\rho^\circ(F_v^\times)$ along the proof of the proposition.

\begin{proof}
Here we employ Fourier transform over $F_v^\times$ to show that $\cS_\rho^\circ(F_v^\times)$ is stable under Hankel transform.
In fact, the connected component of the unitary dual of $F_v^\times$ are indexed by the characters
\[
\chi_n(z)=\left\{\begin{array}{cc}
&(\frac{z}{|z|_v^{1/2}})^n, n\in \Z, \text{ if } F_v=\C,\\
&(\frac{z}{|z|_v})^n, n\in \Z/2\Z,  \text{ if } F_v=\R.
\end{array}\right.
\]
Note that we have 
\begin{align*}
\cM(\bL_v)(\chi, \mu)&=
\int_{F_v^\times }\chi(u)|u|_v^{\mu}\bL_v(u)d^\times u\\
&=\left\{\begin{array}{cc}
\zeta_{F_v}(\mu)\zeta_{F_v}(1-2s+\mu) \text{ if }\chi=1,\\
0, \quad \text{ otherwise.}
\end{array}\right.
\end{align*}
As a consequence if $F_v=\R$ and $m\in \Z/2\Z$, we have
\begin{align*}
\cM(u^n\bL_v)(\chi_m, \mu)&=
\int_{F_v^\times }\chi_m(u)|u|_v^{\mu}u^n\bL_v(u)d^\times u\\
&=\left\{\begin{array}{cc}
\cM(\bL_v)(\mu+n) \text{ if }n=m (\mod 2),\\
0, \quad \text{ otherwise,}
\end{array}\right.
\end{align*}

Similarly,  if $F_v=\C$ and $m\in \Z$, we have for $n_1, n_2\geq 0$, 
\begin{align*}
\cM(u^{n_1} \bar{u}^{n_2}\bL_v)(\chi_m, \mu)&=
\int_{F_v^\times }\chi_m(u)|u|_v^{\mu}u^{n_1} \bar{u}^{n_2}\bL_v(u)d^\times u\\
&=\left\{\begin{array}{cc}
\cM(\bL_v)(\mu+(n_1+n_2)/2) \text{ if }m=n_2-n_1,\\
0, \quad \text{ otherwise.}
\end{array}\right.
\end{align*}

Furthermore, we have for $F_v=\R$, 
\begin{align*}
\cM(u^n\partial\bL_v)(\chi_m, \mu)&=
\int_{F_v^\times }\chi_m(u)|u|_v^{\mu}u^n\partial\bL_v(u)d^\times u\\
&=\left\{\begin{array}{cc}
-(\mu+n)\cM(\bL_v)(\mu+n) \text{ if }n=m (\mod 2),\\
0, \quad \text{ otherwise,}
\end{array}\right.
\end{align*}
and if $F_v=\C$ and $m\in \Z$, we have for $n_1, n_2\geq 0$, 
\begin{align*}
&\cM(u^{n_1} \bar{u}^{n_2}\partial\bL_v)(\chi_m, \mu)\\
&=\int_{F_v^\times }\chi_m(u)|u|_v^{\mu}u^{n_1} \bar{u}^{n_2}\partial\bL_v(u)d^\times u\\
&=\left\{\begin{array}{cc}
-(\mu+(n_1+n_2)/2)\cM(\bL_v)( \mu+(n_1+n_2)/2)\text{ if }m=n_2-n_1,\\
0, \quad \text{ otherwise.}
\end{array}\right.
\end{align*}
For $F_v=\R$, it is clear that under the decomposition into even and odd part: $\C[u]=\C[u]_{e}\oplus \C[u]_{o}$
we have isomorphism of vector spaces
\[
\C[u]_{e}^2\rightarrow \C[\mu], \quad (P_1, P_2)\mapsto \frac{\cM(P_1\bL_v+P_2\partial\bL_v)(\mu)}{\cM(\bL_v)(\mu)}
\]
\[
\C[u]_{o}^2\rightarrow \C[\mu], \quad (P_1, P_2)\mapsto \frac{\cM(P_1\bL_v+P_2\partial\bL_v)(\chi_1, \mu)}{\cM(\bL_v)(\mu+1)}.
\]
As for $F_v=\C$, let us introduce the weight on $\C[z, \bar{z}]: \wt(z)=-1, \wt(\bar{z})=1$.
Let $\C[z, \bar{z}]_m$ be the weight $m$-part of $\C[z, \bar{z}]$.
It follows from the above computation that we have isomorphism of vector spaces
\[
\C[z, \bar{z}]_m^2\rightarrow \C[\mu],  \quad (P_1, P_2)\mapsto \frac{\cM(P_1\bL_v+P_2\partial\bL_v)(\chi_m, \mu)}{\cM(\bL_v)(\mu+m/2)}, \quad \text{ if } m\geq 0.
\]
\[
\C[z, \bar{z}]_m^2\rightarrow \C[\mu],  \quad (P_1, P_2)\mapsto \frac{\cM(P_1\bL_v+P_2\partial\bL_v)(\chi_m, \mu)}{\cM(\bL_v)(\mu-m/2)}, \quad \text{ if } m\leq 0.
\]
Finally, let us compute the Fourier(Mellin) transform of  $\cH_s(f)$ for $f\in \cS_\rho^\circ(M_\rho(F_v))$. We have 
by Proposition \ref{prop-local-functional-equation-tame-schwartz},
\begin{align}\label{eqn-local-functional-equation-tame-schwartz}
&\cM(\cH_s(f))(\chi, \mu)\\
&=\gamma_v(2s-\mu, \chi^{-1})\gamma_v(1-\mu, \chi^{-1})\cM(f)(\chi^{-1}, 2s-\mu)\nonumber\\
&=\epsilon(1-\mu, \chi^{-1})\epsilon(2s-\mu, \chi^{-1})\frac{L_{v}(\mu, \chi)L_{v}(1-2s+\mu, \chi)}{L_{v}(1-\mu, \chi^{-1})L_{v}(2s-\mu,\chi^{-1})}\cM(f)(\chi^{-1}, 2s-\mu)\nonumber.
\end{align}
Of course, this is just the local functional equation for Hankel transform.
Note that for $F_v=\C$, if $\chi=\chi_n$, we have(cf. \cite[Proposition 3.5]{Kud04})
\begin{align*}
L_v(\mu, \chi_n)=\left\{
\begin{array}{cc}
&\zeta_v(\mu-n/2),  \text{ if } n\leq 0,\\
&\zeta_v(\mu+n/2), \text{ if } n\geq 0.
\end{array}\right.
\end{align*}
Hence
\begin{align*}
&\cM(\cH_s(f))(\chi_n, \mu)=\\
&\left\{\begin{array}{cc}
&\epsilon(1-\mu, \chi_{-n})\epsilon(2s-\mu, \chi_{-n})\frac{\cM(\bL_v)(\mu+n/2)}{\cM(\bL_v)(2s-\mu+n/2)}\cM(f)(\chi_{-n}, 2s-\mu), n\geq 0\\
&\epsilon(1-\mu, \chi_{-n})\epsilon(2s-\mu, \chi_{-n})\frac{\cM(\bL_v)(\mu-n/2)}{\cM(\bL_v)(2s-\mu-n/2)}\cM(f)(\chi_{-n}, 2s-\mu), n\leq 0.
\end{array}\right.
\end{align*}
For $F_v=\R$ and $\chi=\chi_n$, we also have $L_v(\mu, \chi_n)=\zeta_v(s+n)(n=0,1)$.  And
\[
\cM(\cH_s(f))(\chi_n, \mu)=\epsilon(1-\mu, \chi_{-n})\epsilon(2s-\mu, \chi_{-n})\frac{\cM(\bL_v)(\mu+n)}{\cM(\bL_v)(2s-\mu+n)}\cM(f)(\chi_{n}, 2s-\mu).
\]

Now one applies the fact that the archimedean epsilon factor $\epsilon(\mu, \chi)$ is independent of $\mu$
and depends only on $\chi$ which is always a power of $\sqrt{-1}$ (cf. \cite[Proposition 3.8]{Kud04}).
Combining our explicit computation shows that $\cM(\cH_s(f))$ is in the 
image of $\cM(\cS_\rho^\circ(M_\rho(F_v)))$.
Finally,  notice that the space $\cS^\circ_\rho(M_\rho(F_v))$ is contained in $H^{\delta}(F_v, dx)(0\leq \delta<3/2-2\Re(s))$ by Corollary \ref{cor-extension-definition-Hankel-transform}.
Therefore by Proposition \ref{prop-embedding-sobolev} we always have $\cH_s(f)\in \cS^\circ_\rho(M_\rho(F_v))$.
\end{proof}

\remk In fact, one can explicitly determine the image of $\cH_s$ on a basis without using Fourier transform.

\begin{cor}\label{cor-Hankel-transform-basis}
We have for $F_v=\C$, 
\begin{align*}
\cH_s(u^n\bL_v)&=(-1)^n\bar{u}^n\bL_v, \\
\cH_s(\bar{u}^n\bL_v)&=(-1)^nu^n\bL_v.
\end{align*}
For $F_v=\R$, 
\[
\cH_s(u^n\bL_v)=(-1)^n u^n\bL_v.
\]
\end{cor}
\remk Combining with Lemma \ref{cor-Hankel-commutes-with-derivative}, the map $\cH_s$ on $ \cS^\circ_\rho(M_\rho(F_v))$ are completely determined.
\\

\begin{prop}\label{prop-density-tame-schwartz-space}
Assume that $v$ is archimedean and $\Re(s)\leq 1/2$. The space $\cS_\rho^\circ(F_v^\times)$ is dense in $L^2(F_v, dx)$.
\end{prop}

\begin{proof}
Assume that $F_v=\R$. It suffices to show that if $f\in L^2(F_v, dx)$ and
\begin{equation}\label{eqn-integral-against-poly}
\int_{F_v^\times}f(x)P(u, \partial)(\bL_v)(x) dx=0
\end{equation}
for each polynomial $P$ then $f(x)=0$ almost everywhere. 
The Plancherel's theorem for Mellin transform implies that for $g_1, g_2\in L^2(F_v, dx)$,
\begin{align*}
\int_{F_v}g_1(x)\overline{g_2(x)}dx&=\int_{\R} \cM(g_1)(1/2+i t)\overline{\cM(g_2)(1/2+it)} dt\\
&+\int_{\R} \cM(g_1)(\chi_1, 1/2+i t)\overline{\cM(g_2)(\chi_1, 1/2+it)} dt.
\end{align*}
Note that if we write $P(u, \partial)\bL_v=P_1(\partial)\bL_v+P_2(\partial)u\bL_v$, 
following the computation in the proof of Proposition \ref{prop-stability-tame-Schwartz}, we have 
\begin{align*}
\cM(P(u, \partial)(\bL_v))(\mu)&=
\int_{F_v^\times }|u|_v^{\mu}P_1(\partial)(\bL_v)(u)d^\times u=\cM(\bL_v)(\mu)P_1(-\mu),\\
\cM(P(u, \partial)(\bL_v))(\chi_1, \mu)&=
\int_{F_v^\times }|u|_v^{\mu}P_2(\partial)(u\bL_v)(u)d^\times u=\cM(u\bL_v)(\mu)P_2(-\mu).
\end{align*}
Hence we can rewrite (\ref{eqn-integral-against-poly}) as
\begin{align}
\int_{\R}\overline{\cM(f)(1/2+it)}\cM(\bL_v)(1/2+it)P_1(-1/2-it) dt&=0,\label{eqn-vanishing-L2-integral-trivial-component} \\
\int_{\R}\overline{\cM(f)(\chi_1,1/2+it)}\cM(\bL_v)(\chi_1, 1/2+it)P_1(-1/2-it) dt&=0.
\end{align}
Let us show how to deduce $\cM(f)(1/2+it)=0$ from (\ref{eqn-vanishing-L2-integral-trivial-component}). 
Let us recall the theorem of Hamburger(cf. \cite[Corollary 2.4]{Schm20}): let $dm_N(x)$ be a positive Borel measure over $\R^N$, if 
$\int_{\R^N}e^{c|x|}dm_N(x)<\infty$ for some $c$, then the polynomials are dense in $L^2(\R^N, dm_N(x))$.
In our case, we take
\[
dm(t)=|\cM(\bL_v)(1/2+it)|dt
\]
which satisfies the assumption of Hamburger's theorem due to the Stirling estimate:
\begin{equation}\label{eqn-stirling-formula}
|\Gamma(x+iy)|=\sqrt{2\pi}|y|^{x-1/2}e^{-\pi |y|/2}[1+O(1/y)], \quad |y|\rightarrow \infty, 
\end{equation}
and (cf. \cite[(1.5)]{Yaku12})
\[
|\Gamma(z)|\leq |\Gamma(\Re(z))|.
\]
Let
\[
h(t)=\overline{\cM(f)(1/2+it)}\frac{\cM(\bL_v)(1/2+it)}{|\cM(\bL_v)(1/2+it)|}.
\]
Then
\[
\int_{\R}|h(t)|^2dm(t)\leq C\int_{\R}|\cM(f)(1/2+it)|^2 dt, \quad C\in \R.
\]
Now apply Hamburger's theorem to deduce $g(t)=0$. Hence $\cM(f)(1/2+it)=0$. Similar proof works to prove $\cM(f)(\chi_1, 1/2+it)$=0.
Hence $f(x)=0$ for almost all $x$.

Now consider the case $F_v=\C$. 
Then the connected components of the unitary dual of $F_v^\times$ is indexed by $\chi_n=(\frac{z}{|z|})^n$.
Let $g=P(z, \bar{z}, \partial)\bL_v\in \cS_\rho^\circ(M_\rho(F_v))$ as
in (\ref{eqn-integral-against-poly}), 
we still have $g=\sum_{m> 0} P_m(\partial)z^m\bL_v+\sum_{m\geq 0} Q_m(\partial)\bar{z}^m\bL_v$. And
following the computation in the proof of Proposition \ref{prop-stability-tame-Schwartz}, we have, 
\begin{align*}
\cM(g)(\chi_m, \mu)&=
\int_{F_v^\times }|u|_v^{\mu}\chi_m(u)g(u)d^\times u=\cM(\bar{z}^m\bL_v)(\mu)Q_m(-\mu)(m\geq 0),\\
\cM(g)(\chi_m, \mu)=&\int_{F_v^\times }|u|_v^{\mu}\chi_m(u)g(u)d^\times u=\cM(z^{-m}\bL_v)(\mu)P_{-m}(-\mu)(m< 0).
\end{align*}

Now Plancherel's theorem implies that 
\[
\sum_{m}\int_{\R}\overline{\cM(f)(1/2+it)}\cM(g)(\chi_m, 1/2+it) dt=0, 
\]
the left hand side is a finite sum as we remark before. 
Now arguing as in the case $F_v=\R$ gives the result.

\end{proof}

\begin{prop}
The Hankel transform is self-inversive: for $f\in \cS_\rho(F_v^\times)$, 
\[
\cH_s^2(f)(y)=f(y).
\]
\end{prop}

\begin{proof}
Note that we already know that $\cH_s$ is an isometry of $H^{1-2s}(F_v^\times)$ by Proposition \ref{prop-embedding-sobolev}
and that $\cS_\rho(F_v^\times)\subseteq H^{1-2s}(F_v^\times)$ by Corollary \ref{cor-extension-definition-Hankel-transform}.
We apply Proposition \ref{lem-Hankel-Fourier}. For $f\in \cS_\rho(F_v^\times)$,
\begin{equation}\label{eqn-check-involutive-Hankel}
\cH_s^2(f)=(|x|_v^{2s-2}\hat{\cH_s(f)}(1/x))^{\hat{}}.
\end{equation}
Since
\[
\hat{\cH_s(f)}(x)=|x|_v^{2s-2} \hat{f}(1/x),
\]
we obtain 
\[
\cH_s^2(f)=f.
\]
\end{proof}

\remk Functions in $\cS_\rho(F_v^\times)$ are called special  $\rho$-functions in \cite{Laff16}. It is a small space 
(although it is dense in $L^2(F_v, dx)$) in the sense that it does not contain functions of compactly supported. 

\remk To conclude this section, we note that we do not show that the space $\cS_\rho(F_v)$ is stable under Hankel transform. 
This could be done by analyzing the image of $\cS_\rho(F_v)$ under Mellin transform(i.e., to characterize the aymptotics of 
$\cM(\phi)$ for $\phi\in \cS_\rho(F_v)$ following \cite[\S 1.2.1]{Schu91}. However, we feel that this would bring us too far. We
hope to come back to this question in the future.

\section{Poisson Summation}

We keep the same notation as in previous section.

\begin{teo}[Poisson summation]\label{teo-Poisson-summation-GL_1}
For any $\phi\in \cS_\rho(\A_F^\times)$ we have 
\[
\sum_{a\in F^\times} \phi(a)=\sum_{a\in F^\times}\cH_s{\phi}(a)+\text{other terms}
\]
\end{teo}
\remk The full formula including the extra terms will be determined in the process of the proof.

We first recall the special Eisenstein series we need and then deduce the Poisson summation formula 
from the modular relation of Eisenstein series. This is the classical approach to Voronoi type summation formula
taken up by Bochner(\cite{Boch51}).

Our construction of Eisenstein series follows \cite[\S 9]{Bump06}. 
Let $\phi=\otimes_v \phi_v\in \cS_\rho(\A_F^\times)$.
From $\phi$ we construct
$f_\phi=\otimes_{v} f_{\phi, v}$ as follows. Let $B$ be the Borel subgroup of $\GL_2$ defined over $F$ which consists of lower triangular matrices. 
Consider the Bruhat decompostion of $\GL_2=B\coprod BwB$ with $w=\begin{bmatrix}0&1\\1&0\end{bmatrix}$.
First of all, we require that 
\[
f_\phi(x\begin{bmatrix}t_1& 0\\ \star & t_2\end{bmatrix}, s)=|\frac{t_1}{t_2}|^{1-s}f_\phi(x, s).
\]
Secondly, we let 
\[
f_{\phi, v}(\nu(u), s)=\frac{1}{\zeta_{F_v}(2-2s)}\int_{F_v}\phi_v(x, s)\psi_v(-ux) dx.
\]
By Bruhat decomposition, this uniquely defines $f$.
Note that for all but finitely many $v$, $\phi_v(u, s)=\bL_v(u, s)$ is basic. In this case
we have 
\[
f_{\phi, v}(\nu(x), s)=t(\nu(x))^{2-2s}.
\]
Note that by Corollary \ref{cor-decaying-schwartz-function}, $f_v(\nu(x), s)$ is of decay order 
$|x|_{v}^{2\Re(s)-2}$ as $|x|_v\rightarrow \infty$ for $v$ archimedean.

Finally, by Fourier inversion formula
\[
\phi_v(u, s)=\zeta_{F_v}(2-2s)\int_{F_v}f_{\phi, v}(\nu(x), s)\psi_v(ux) dx.
\]

We follow \cite[\S 6]{Wri85} to define our Eisenstein series. Then we define
\[
E(f_\phi,s,  g)=\sum_{\gamma\in \GL_2(F)/B(F)}f_\phi(g\gamma, s), \quad g\in \GL_2(\A_F).
\]
The basic properties are sumarized 
in the following. For related results of Eisenstein series, we refer \cite[Lemma 6.1]{Wri85}, \cite[\S 6.2]{Shah10}.

\begin{prop}\label{prop-Eisenstein-series-fundamental-property}\noindent
\begin{itemize}
\item[(i)]$E(f_\phi, s, g)$ converges absolutely and locally uniformly for $\Re(s)<0$ and $g\in G(\A_F)$.

\item[(ii)] For $\Re(s)<1/2$, the Eisenstein series has the following convergent Fourier 
expansion: 
\[
E(f_\phi, s, a(t_1, t_2)n(x_0))=a_0(\frac{t_1}{t_2})+\frac{|\frac{t_1}{t_2}|^s}{Z_{F}(2-2s)}\sum_{\alpha\in F^\times}\cH_s\phi(\alpha t_1/t_2, s)
\psi_{\A_F}(\alpha x_0),
\]
where 
\[
a_0(t)=|t|^{1-s} f_\phi(\Id, s)+\frac{|t|^s}{Z_{F}(2-2s)}\int_{\A_F}|u|^{2s-1}\phi(u, s)du, 
\]
with $Z_{F}(s)$ the completed zeta function of $F$
\[
Z_F(s)=\prod_{v}\zeta_v(s).
\]
\item[(iii)] For $\Re(s)<1/2$, the Eisenstein series has the following convergent Fourier 
expansion: 
\[
E(f_\phi, s, wa(t_1, t_2)n(x_0))=b_0(\frac{t_1}{t_2})+\frac{|\frac{t_1}{t_2}|^s}{Z_{F}(2-2s)}\sum_{\alpha\in F^\times}\phi(-\alpha t_1/t_2, s)
\psi_{\A_F}(\alpha x_0),
\]
where 
\[
b_0(t)=|t|^{1-s} f_\phi(w, s)+\frac{|t|^{s}}{Z_{F}(2-2s)}\phi(0, s).
\]
\end{itemize}

\end{prop}

\begin{proof}
The convergence follows from the general theorems cited above.
Let us determine the constant term. 
We have 
\begin{align*}
E_0(f_\phi, s, a(t_1, t_2)n(x_0))&=f_\phi(a(t_1, t_2)n(x_0), s)+\int_{\A_F}f_\phi(a(t_1, t_2)n(x+x_0)w, s)dx\\
&=|\frac{t_1}{t_2}|^{1-s} f_\phi(\Id, s)+\int_{\A_F}f_\phi(a(t_1, t_2)n(x)w, s)dx.
\end{align*}
Since
\[
a(t_1, t_2)n(x)w=\begin{bmatrix}0&t_1\\ t_2& t_2 x\end{bmatrix}=\begin{bmatrix}1&t_1t_2^{-1}x^{-1}\\ 0& 1\end{bmatrix}
\begin{bmatrix}-x^{-1}t_1&0\\ t_2& t_2x\end{bmatrix}
\]
we have 
\[
f_\phi(a(t_1, t_2)n(x)w, s)=|\frac{t_1}{x^2t_2}|^{1-s}f_\phi(\nu(t_1t_2^{-1}x^{-1}), s).
\]
Hence 
\begin{align*}
\int_{F_v}f_{\phi, v}(a(t_1, t_2)n(x)w, s)dx&=|\frac{t_1}{t_2}|_v^{1-s}\int_{F_v}|x|_v^{2s-2} f_{\phi, v}(\nu(t_1t_2^{-1}x^{-1}, s))dx\\
&=|\frac{t_1}{t_2}|_v^s\int_{F_v}|x|_v^{-2s}f_{\phi, v}(\nu(x), s)dx.
\end{align*}
Moreover, we know 
\begin{align*}
\int_{F_v}|x|_v^{-2s}\psi_v(ux)dx&=|u|_v^{2s-1}\gamma_v(2s),\\
\int_{F_v}f_{\phi, v}(\nu(x), s)\psi_v(ux)dx&=\frac{\phi_v(u, s)}{\zeta_{F_v}(2-2s)}.
\end{align*}

Using the fact that Fourier transform sends product to convolution, we have 
\[
\int_{F_v}|x|_v^{-2s}f_{\phi, v}(\nu(x), s)dx=\frac{\gamma_v(2s)}{\zeta_{F_v}(2-2s)}\int_{F_v}|u|_v^{2s-1}\phi_v(u, s)du.
\]
Note that when $v$ is non-archimedean and $\phi_v(u, s)$ is basic (cf. (\ref{eqn-definition-basic-function-Eisenstein})), i.e., 
\begin{align*}
\phi_v(u, s)=\left\{\begin{array}{ll}
&|\mathfrak{D}_v|_v^{s-1/2}\sum_{i=0}^{\nu_v(u\omega_v)}q_v^{i(2s-1)}, \text{ if } \nu_v(u\omega_v)\geq 0,\\
&0, \quad \text{ otherwise }
\end{array}\right.
\end{align*}
we have 
\[
\int_{F_v}|u|_v^{2s-1}\phi_v(u, s)du=\zeta_{F_v}(2s).
\]

Back to genera case, taking product, we obtain 
\[
E_0(f_\phi, s, a(t_1, t_2))=|\frac{t_1}{t_2}|_v^{1-s} f_\phi(\Id, s)+\frac{|\frac{t_1}{t_2}|_v^{s}}{Z_{F}(2-2s)}\int_{\A_F}|u|_v^{2s-1}\phi(u, s)du.
\]
Similarly, for positive terms we have 
\begin{align*}
&\int_{F_v}f_{\phi, v}(a(t_1, t_2)n(x+x_0)w, s)\psi_v(-ux)dx\\
&=|\frac{t_1}{t_2}|^{1-s}\psi_v(ux_0)\int_{F_v}|x|_v^{2s-2} f_{\phi, v}(\nu(t_1t_2^{-1}x^{-1}, s))\psi_v(-ux)dx\\
&=|\frac{t_1}{t_2}|_v^s\psi_v(ux_0)\int_{F_v}|x|_v^{-2s}f_{\phi, v}(\nu(x), s)\psi_v(-ut_1t_2^{-1}x^{-1})dx.
\end{align*}
Plug in the formula 
\[
f_{\phi, v}(\nu(x), s)=\frac{1}{\zeta_{F_v}(2-2s)}\int_{F_v}\phi_v(z, s)\psi_v(-zx) dz,
\]
we obtain
\begin{align*}
&\int_{F_v}f_\phi(a(t_1, t_2)n(x+x_0)w, s)\psi_v(-ux)dx\\
&=\frac{|\frac{t_1}{t_2}|_v^s\psi_v(ux_0)}{\zeta_{F_v}(2-2s)}\int_{F_v}\int_{F_v}|x|_v^{-2s}\phi_v(z,s)\psi_v(-zx-ut_1t_2^{-1}x^{-1})dzdx\\
&=\frac{|\frac{t_1}{t_2}|_v^s\psi_v(ux_0)}{\zeta_{F_v}(2-2s)}\int_{F_v}\int_{F_v}|w|_v^{-2s}|z|_v^{2s-1}\phi_v(z,s)\psi_v(-w-ut_1t_2^{-1}zw^{-1})dzdw\\
&=\frac{|\frac{t_1}{t_2}|_v^s\psi_v(ux_0)}{\zeta_{F_v}(2-2s)}\cH_s\phi_v(ut_1t_2^{-1}).
\end{align*}
Taking product over all places yields the result.

We continue to prove (iii).

Let us determine the constant term. 
We have 
\begin{align*}
E_0(f_{\phi}, s, wa(t_1, t_2)n(x_0))&=f_\phi(wa(t_1, t_2)n(x_0), s)+\int_{\A_F}f_\phi(wa(t_1, t_2)n(x+x_0)w, s)dx\\
&=|\frac{t_1}{t_2}|^{1-s} f_\phi(w, s)+\int_{\A_F}f_\phi(wa(t_1, t_2)n(x)w, s)dx.
\end{align*}
Since
\[
wa(t_1, t_2)n(x)w=\begin{bmatrix}t_2&t_2x\\ 0& t_1 \end{bmatrix}=\begin{bmatrix}1&t_2t_1^{-1}x\\ 0& 1\end{bmatrix}
\begin{bmatrix}t_2&0\\ 0& t_1\end{bmatrix}
\]
we have 
\[
f_\phi(wa(t_1, t_2)n(x)w, s)=|\frac{t_2}{t_1}|^{1-s}f_\phi(\nu(t_2t_1^{-1}x), s).
\]
Hence 
\begin{align*}
\int_{F_v}f_{\phi, v}(wa(t_1, t_2)n(x)w, s)dx&=|\frac{t_2}{t_1}|_v^{1-s}\int_{F_v}f_{\phi, v}(\nu(t_2t_1^{-1}x, s))dx\\
&=|\frac{t_1}{t_2}|_v^{s}\int_{F_v}f_{\phi, v}(\nu(x), s)dx.
\end{align*}
Moreover, we know 
\begin{align*}
\int_{F_v}f_{\phi, v}(\nu(x), s)\psi_v(ux)dx&=\frac{\phi_v(u, s)}{\zeta_{F_v}(2-2s)}.
\end{align*}

We deduce that 
\[
\int_{F_v}f_{\phi, v}(wa(t_1, t_2)n(x)w, s)dx=|\frac{t_1}{t_2}|_v^{s}\frac{\phi_v(0, s)}{\zeta_{F_v}(2-2s)}.
\]
Taking product, we obtain 
\[
E_0(f_\phi, s, a(t_1, t_2)n(x))=|\frac{t_1}{t_2}|^{1-s} f_\phi(w, s)+\frac{|\frac{t_1}{t_2}|^{s}}{Z_{F}(2-2s)}\phi(0, s).
\]

Similarly, for positive terms we have 
\begin{align*}
&\int_{F_v}f_\phi(wa(t_1, t_2)n(x+x_0)w, s)\psi_v(-ux)dx\\
&=|\frac{t_2}{t_1}|_v^{1-s}\psi_v(ux_0)\int_{F_v}f_{\phi, v}(\nu(t_2t_1^{-1}x, s))\psi_v(-ux)dx\\
&=|\frac{t_1}{t_2}|_v^s\psi_v(ux_0)\int_{F_v}f_{\phi, v}(\nu(x), s)\psi_v(-ut_2^{-1}t_1x)dx.
\end{align*}
Plug in the formula 
\[
\phi_v(u, s)=\zeta_{F_v}(2-2s)\int_{F_v}f_{\phi, v}(\nu(x), s)\psi_v(ux) dx.
\]
we obtain
\begin{align*}
&\int_{F_v}f_{\phi, v}(wa(t_1, t_2)n(x+x_0)w, s)\psi_v(-ux)dx\\
&=\frac{|\frac{t_1}{t_2}|_v^s\psi_v(ux_0)}{\zeta_{F_v}(2-2s)}\phi_v(-ut_1t_2^{-1}, s)
\end{align*}
Taking product over all places yields the result.
\end{proof}

\begin{proof}[Proof of Theorem \ref{teo-Poisson-summation-GL_1}]
From Proposition \ref{prop-Eisenstein-series-fundamental-property} we have
\[
E(f_\phi, s, \Id)=a_0(1)+\frac{1}{Z_{F}(2-2s)}\sum_{\alpha\in F^\times}\cH_s{\phi}(\alpha , s),
\]
and 
\[
E(f_\phi, s, w)=b_0(1)+\frac{1}{Z_{F}(2-2s)}\sum_{\alpha\in F^\times}\phi(\alpha , s).
\]

But we have $E(f_\phi, s, \Id)=E(f_\phi, s, w)$, this gives 
\[
Z_{F}(2-2s)a_0(1)+\sum_{\alpha\in F^\times}\cH_s{\phi}(\alpha , s)=Z_{F}(2-2s)b_0(1)+\sum_{\alpha\in F^\times}\phi(\alpha , s), 
\]
which is our desired Poisson summation formula.
\end{proof}

\remk We can further evaluate 
\begin{align*}
f_\phi(\Id, s)&=\lim_{u\rightarrow 0}f_\phi(\nu(u), s)=\lim_{u\rightarrow 0}\frac{1}{Z_F(2-2s)}\int_{\A_F}\phi(x, s)\psi(-ux) dx, \\
f_\phi(w, s)&=\lim_{u\rightarrow 0}f_\phi(\nu(u)w, s)=\lim_{u\rightarrow 0}\frac{1}{Z_F(2-2s)}|u|^{2s-2}\int_{\A_F}\phi(x, s)\psi(-u^{-1}x) dx.
\end{align*}

\bibliographystyle{plain}
\bibliography{biblio}


\end{document}